\documentclass[1p]{elsarticle}
\usepackage[english]{babel}
\usepackage{amsmath}
\usepackage{amssymb}
\usepackage{mathptmx}     
\usepackage{latexsym}
\usepackage{relsize}
\usepackage{xcolor}
\usepackage{amsthm}
\usepackage{cancel}
\usepackage{hyperref}
\newtheorem{teo}{Theorem}
\newtheorem{lem}{Lemma}
\newtheorem{rem}{Remark}
 
 \newcommand{\adots}
 {\mathinner{\mkern2mu\raise1pt\hbox{.}\mkern2mu\raise4pt\hbox{.}\mkern2mu\raise7pt\hbox{.}\mkern1mu}}
\title{On polynomial solutions of certain finite order\\ ordinary differential equations\tnoteref{t1,t2}}

\begin{document}

\begin{frontmatter}
	\author[1]{L.M. Anguas\fnref{fn1}}
	\ead{luismiguel.anguas@slu.edu}
	\affiliation[1]{organization={Saint Louis University, Madrid Campus},
		addressline={Avenida del Valle 34},
		postcode={28003},
		city={Madrid},
		country={Spain}}

	\author[2]{D. Barrios Rolan\'{\i}a\corref{cor1}\fnref{fn2}}
	\ead{dolores.barrios.rolania@upm.es}
	\affiliation[2]{organization={ETSI Industriales, Universidad Politecnica de Madrid},
		addressline={C/Jose Gutierrez Abascal 2},
		postcode={28006},
		city={Madrid},
		country={Spain}}
	
	\cortext[cor1]{Corresponding author}
	\fntext[fn1]{The work of this author was partially supported by Agencia Estatal de Investigaci\'on, Ministerio de Ciencia e Innovaci\'on, Spain,
		under grant PID2019-106362GB-I00.}
	\fntext[fn2]{The work of this author was partially supported by Agencia Estatal de Investigaci\'on, Ministerio de Ciencia e Innovaci\'on, Spain,
		under grant PID2021-122154NB-I00.}
		
\begin{abstract}
	Some properties and relations satisfied by the polynomial solutions of a bispectral problem are studied. Given a finite order differential operator, under certain restrictions, its polynomial eigenfunctions are explicitly obtained, as well as the corresponding eigenvalues. Also, some linear transformations are applied to sequences of eigenfunctions and a necessary condition for this to be a sequence of eigenfunctions of a new differential operator is obtained. These results are applied to the particular case of classical Hermite polynomials.
\end{abstract}


		
\end{frontmatter}



\section{Introduction and main results}
We consider the ordinary differential operator of order $N$
\begin{equation}
L\equiv                
\sum_{i=0}^Na_i(x)\partial_x^i
\label{1}
\end{equation}
where 
$a_i(x)
$
are polynomials in the variable $x$,
$
\deg (a_i)\leq i,
$ and 
$
\partial^i_x,\, i=1,\ldots,N$, represents the derivative of order $i$ with respect to $x$. 
We also consider a sequence 
$
\{\lambda_n\}\subset \mathbb{C}
$
of eigenvalues and the corresponding sequence of eigenfunctions
$
\{P_n\},
$
which we assume that are monic polynomials with $\deg(P_n)=n$ for each $n\in \mathbb{N}$. That is,
\begin{equation}        
\sum_{i=1}^Na_i(x)\partial_x^iP_n(x)=\lambda_nP_n(x)\,,\quad \forall n\in \mathbb{N}.
\label{11}
\end{equation}
The polynomials $\{P_n\}$ satisfying \eqref{11} are called {\em eigenpolynomials} in this work. We approach the relations between the operator $L$ and its eigenvalues and eigenpolynomials. 

Although we study eigenpolynomials of \eqref{1} in general, we are especially interested in families of polynomials that are at the same time eigenfunctions of a certain difference operator $J$. One of such sequences  
$
\{P_n\}
$
of polynomials satisfy
\begin{equation}
\label{diferencias}
J
\left(
\begin{array}{c}
P_0(x)\\
P_1(x)\\
\vdots \\
\end{array}
\right)=
x\left(
\begin{array}{c}
P_0(x)\\
P_1(x)\\
\vdots \\
\end{array}
\right),
\end{equation}
where
\begin{equation}
\label{otro13}
J=
\left(
\begin{array}{rcccccccccc}
\alpha_{0,0}&1&0&\cdots  \\ 
 \vdots&\ddots&\ddots&\ddots  \\
 \alpha_{p,0}&\cdots & \alpha_{p,p}&1&0&\cdots \\ 
0&\alpha_{p+1,1}&\cdots&\alpha_{p+1,p+1}&1&0& \cdots\\ 
 & 0&\ddots&  \ddots&  \ddots& \ddots&\\
&&0&\alpha_{n,n-p}&\cdots&\alpha_{n,n}&1&0& \cdots\\ 
 &&&\ddots &\ddots&  \ddots& \ddots& \ddots& \\
 \end{array}
\right).
\end{equation}
Equation \eqref{diferencias} is satisfied if and only if those polynomials  satisfy a $(p+2)$-term recurrence relation
\begin{equation}
\label{otro9}
\displaystyle \sum_{k=n-p}^{n-1}\alpha_{n,k}P_k(x)+(\alpha_{n,n}-x)P_n(x)+P_{n+1} = 0\,,\quad n=0,1,\ldots ,
\end{equation}
with initial conditions 
$$
P_0=1,\quad P_{-1}=\cdots =P_{-p}=0.
$$
The difference operator is given by 
$$
J(n)P_n=\sum_{k=n-p}^{n}\alpha_{n,k}P_k+P_{n+1}
$$
and we have 
\begin{equation}
\label{discreto}
\left(J(n)P_n\right)(x)=xP_n(x).
\end{equation}
If the polynomials
$
P_n
$ 
satisfy \eqref{11} and \eqref{discreto}, we say that the sequence
$
\{P_n\},\,n\in\mathbb N,
$ 
is a solution for the bispectral problem defined by $L$ and $J$.

S. Bochner \cite{Bochner} studied the above problem for the case where the order of the differential operator is $N=2$ and determined the polynomial solutions of \eqref{11}. He completely solved the problem, and his classification defined the nowadays well-known families of classical orthogonal polynomials which correspond to $p=1$ in \eqref{otro9}. Some years later, Krall \cite{Krall1}, \cite{Krall} studied the differential operator of order $N=4$, giving a classification with seven families. This classification includes three new families of polynomial eigenfunctions that cannot be reduced to operators of order 2. Since the celebrated paper by Bochner and the relevant contributions by Krall, there have been several contributions about the bispectral problem, see, for example, \cite{Shapiro}, \cite{Paco}. Some extensions of these works have been attempted, even for operators in differences with complex coefficients \cite{Grunbaum1}, \cite{Grunbaum2}. However, the difficulty of the problem has made impossible to obtain conclusions as relevant as those already known for the operators orders $N=2$ and $N=4$. The goal of this work is to shed some light on this problem by providing some relationships between the differential operator \eqref{1} and its eigenpolynomials. We hope that these contributions open new possibilities for solving the problem with a general order $N$.

In what follows, we assume that the differential operator \eqref{1} is given; we also assume that there exist both the sequence of eigenvalues 
$\{\lambda_n\}$
and the corresponding sequence of eigenpolynomials
$
\{P_n\} 
$
satisfying \eqref{11} . We may assume $a_0\equiv 0$ because, otherwise, we would substitute $a_0$ by $a_0-\lambda_0$. For the same reason, we take $\lambda_0=0$. Also, we define 
\begin{equation}
a_n(x)=0,\quad n>N.
\label{aaa}
\end{equation}
Then, for each 
$n\in \mathbb{N}$
we write the polynomials $a_n(x)$ and $P_n(x)$ as
\begin{equation}
a_n(x)=\sum_{i=0}^n a_{n,i}x^i,\quad a_{n,i}\in \mathbb{C},\,i=0,1,\ldots, n,
\label{1111}
\end{equation}
and
\begin{equation}
P_n(x)=\sum_{i=0}^n b_{n,i}x^i,\quad b_{n,n}=1\,, b_{n,i}\in \mathbb{C},\,i=0,1,\ldots, n-1.
\label{111}
\end{equation}

A relevant tool in this paper is the sequence
$\{\delta_n^{(k)}\}$
defined from the coefficients of the polynomials 
$a_i(x),\,i=1,\ldots,N,$
as
\begin{equation}
\label{deltas_aes}
\delta_n^{(k)}=\sum_{i=k}^{n} {n \choose i} i!a_{i,i-k}, \quad  k=0,1,\ldots, n\,.
\end{equation}

The following theorem is the key to understand the connection between these sequences \eqref{deltas_aes} and the coefficients of the eigenpolynomials. We remark that, in addition, this theorem provides a valuable method to obtain the coefficients $b_{n,k},\,k=0,1,\ldots ,n-1,$ from the polynomials $P_n(x)$ (see \eqref{111}).

\begin{teo}
\label{lema1}
For the sequence 
$
\{P_n(x)\}
$
of eigenpolynomials and the sequence 
$
\{\lambda_n\}
$
of eigenvalues of $L$ we have
\begin{enumerate}
\item For each $n\in \mathbb{N}$, \eqref{11} is equivalent to
\begin{equation}
\sum_{k=0}^N\delta_{m+k}^{(k)} b_{n,m+k}=\lambda_nb_{n,m},\quad m=0,1,\ldots, n,
\label{7}
\end{equation}
where $b_{n,m+k},\,k=0,\ldots, ,N,$ are the coefficients of the polynomials $\{P_n\}$ as described in \eqref{111}.
\item
Let $M$ be the semi-infinite upper triangular matrix 
$$
M=\left(\begin{array}{ccccccccc}
\delta^{(0)}_{0}&\delta^{(1)}_{1}&\cdots&\cdots&\delta^{(N)}_{N}&0&&&\\
 0&\delta^{(0)}_{1}&\delta^{(1)}_{2}&\cdots&\cdots&\delta^{(N)}_{N+1}&0&&\\
 &0&\delta^{(0)}_{2}&\delta^{(1)}_{3}&\cdots&\cdots&\delta^{(N)}_{N+2}&0&\\
 & &\ddots&\ddots&\ddots&\cdots&&\ddots&\ddots
\end{array}\right)
$$
and let $M_{n+1}$ be the truncation of
$M$
formed by its first 
$n+1$
rows and columns for each fixed $n\in \mathbb{N}$. Assume 
\begin{equation}
\lambda_n\neq 0, \lambda_1,\,\lambda_2,\ldots, \lambda_{n-1},\quad n=1,2,\ldots
\label{autovectores}
\end{equation}
Then $P_n(x),\,n=1,2\ldots,$ is the unique monic polynomial satisfying \eqref{11} whose coefficients $b_{n,0},\cdots, b_{n,n-1}$ in \eqref{111} determine an eigenvector of $M_{n+1}$ corresponding to the eigenvalue $\lambda_n$. That is, 
\begin{equation}
\label{otro11}
(M_{n+1}-\lambda_nI_{n+1})b_n=0,
\end{equation}
where $b_n=(b_{n,0},\cdots, b_{n,n-1},1)^T$. 
\end{enumerate}
\end{teo}
(We understand, here and in the rest of the paper, that
$
b_{n,s}=0
$
when 
$s>n$.)

As a consequence of \eqref{7}, the eigenvalues of $L$ are determined as 
\begin{equation}
\lambda_n=\delta_n^{(0)},\quad n\in \mathbb{N}.
\label{****}
\end{equation}
From \eqref{deltas_aes} and \eqref{****}, if \eqref{11} holds then the eigenvalues can be obtained as
\begin{equation}
\lambda_n=\sum_{i=1}^{\min\{n,N\}} 
{n \choose i} i!a_{i,i},\quad n=1,2,\ldots
\label{2}
\end{equation}
(it is also a well-known fact in the literature, see \cite{Everitt}, \cite{Krall}). This expression allows us to remark that, given a differential operator $L$, the sequence of eigenvalues $\{\lambda_n\}$ is unique. 

In the proof of Theorem \ref{lema1} we will see that condition \eqref{autovectores} is not neccesary for the existence of eigenpolynomials. However,  under such condition, the relation \eqref{otro11} provides an easy method to obtain the coefficients of these polynomials as coordinates of eigenvectors for a sequence of finite triangular matrices. Moreover, \eqref{autovectores} is also a relevant condition in other places of this work. In particular, we highlight its importance in the expression of the coefficients of the eigenpolynomials in \eqref{coeficientes}. Henceforth, in the sequel we assume that condition \eqref{autovectores} is satisfied. That is, we assume that the eigenvalues 
$\lambda_n,\,n=0,1,\ldots,  $ are all different from zero and also different from each other, which only depends on the leading coefficients
$
a_{n,n}
$
in \eqref{1111}. 

In our next result, we provide the explicit expression of the coefficients of each eigenpolynomial in terms of the elements of the sequence
$
\{\delta_n^{(k)}\}
$. Due to \eqref{deltas_aes}, this result implies the uniqueness of the sequence 
$
\{P_n\}
$
of eigenpolynomials, which is completely determined by the coefficients of polynomials
$
a_i(x)
$
that define the differential operator $L$.

\begin{teo}
\label{teorema_coeficientes}
Under the above conditions and using the notation from \eqref{111}, for each $n\in\mathbb{N}$ we have 
\begin{equation}
\label{coeficientes}
b_{n,i}=\mathlarger\sum_{E_n^{(i)}}\left(\prod_{s=1}^k\frac{\delta^{(i_s)}_{i+i_1+\cdots +i_s}}{\lambda_n-\lambda_{i+i_1+\cdots +i_{s-1}}}\right),\quad i=0,1,\ldots , n-1,
\end{equation}
where the sum is extended to the set
$
E_n^{(i)}=\{(i_1,\ldots,i_k)\in\mathbb{N}^k\,:\, k\in \mathbb N, \,i_1+\cdots +i_k=n-i\}
$
and we understand 
$
i+i_1+\cdots +i_{s-1}=i
$
when 
$s=1 $.
\end{teo}

The above concepts and results allow us to study the linear transformations 
\begin{equation}
\label{999}
P_n^{(1)}=P_n+\gamma_nP_{n-1}
\end{equation}
of the eigenpolynomials 
$
\{P_n\},
$
where the sequence   
$
\{P_n^{(1)}
\}
$
is obtained from a given sequence
$
\{\gamma_n\}\subset \mathbb{C}.
$
We focus our interest in the case of new families
$
\{P_n^{(1)}
\}
$
which are eigenpolynomials of some finite order differential operator and also satisfy a $(p+2)$-term recurrence relation. In the following theorem, a necessary condition for \eqref{999} to provide a new family of eigenpolynomials for some finite order operator is given. 

\begin{teo}
\label{teorema1}
Assume 
$
\{P_n^{(1)}
\}
$
and
$
\{\lambda_n\}
$
are eigenpolynomials and eigenvalues, respectively, of some finite order differential operator 
\begin{equation}
L^{(1)}=\sum_{i=0}^{\widetilde N}a^{(1)}_i(x)\partial_x^i
\label{operador_nuevo}
\end{equation}
of order 
$\widetilde N\geq N$.
Then 
\begin{equation}
\label{la_del_teorema1}
\sum_{j=0}^{k-1}(-1)^j\left[\sum_{s=1}^N {{n-k+j} \choose {s-1}} s!a_{s,s}\right]b_{n-k+j,n-k}
\sum_{r=1}^{k-j}\gamma_{n-k+j+1}\ldots \gamma_{n-k+j+r} E_{k,n,j+r-1}
=0
\end{equation}
for any $n,\,k\in \mathbb{N}$ such that 
$n\geq k> \widetilde N,$
where
\begin{equation}
\label{otro17}
 E_{k,n,s}=
\left|
\begin{array}{ccccc}
b_{n-k+s+2,n-k+s+1}& b_{n-k+s+3,n-k+s+1}&\cdots&\cdots  & b_{n,n-k+s+1}\\ 
1& b_{n-k+s+3,n-k+s+2}&\cdots &\cdots& b_{n,n-k+s+2}\\ 
0 & 1&   \cdots& \cdots &b_{n,n-k+s+3} \\
\vdots  & 0  & \ddots &&\vdots \\
\vdots  &\vdots & \ddots &\ddots &\vdots \\
0 & 0 &\cdots & 1& b_{n,n-1}
 \end{array}
\right|.
\end{equation}
\end{teo}
We point out the relevance of Theorem \ref{teorema1}, since it makes unnecessary to study coefficients 
$
a^{(1)}_i(x)
$
and values of 
$\widetilde N$,
reducing the problem of non existence to the verification of a condition that depends only on the sequence $\{\gamma_n\}$, as well as the initial operator \eqref{1} and its eigenpolynomials $\{P_n\}$.

In the last part of this work, we will study a particular case of the linear transformation \eqref{999}. More precisely, we will focus on the Geronimus transformations, a particular case of Darboux transformations (see \cite{Geronimus}). Depending on the field of application, different versions of this kind of transformations have been used in the literature  (see the introduction in \cite{MaribelPaco} for more details). We will use the definitions introduced in \cite{IntSys2}, which have also been used in \cite{IntSys1, IntSys3, Pade, Dani}. For the convenience of the reader, Section \ref{seccion3} will include a brief summary of these concepts. In \cite{Grunbaum3}, \cite{Grunbaum1}, the application of Geronimus transformations, under certain conditions, to some classical polynomials was studied. More precisely, in \cite{Grunbaum3} it was showed that the Krall polynomials can be obtained from some instances of the Laguerre and the Jacobi polynomials under these transformations. Moreover, in \cite{Grunbaum1}, the Laguerre polynomials were studied, and the transformation of these polynomials in eigenfunctions of a finite order differential operator under several applications of Geronimus transformation was proved. The Hermite polynomials, in combination with the Bessel polynomials, were only partially studied in \cite{Grunbaum4}. In our paper, we complement these works by showing that the Geronimus transformation on the Hermite polynomials does not produce a new family of eigenfunctions for any finite order differential operator. This will be proved using Theorem \ref{teorema1}, that is, we will prove that condition \eqref{la_del_teorema1} is not satisfied.

The rest of the paper is distributed as follows. In Section \ref{seccion2} we analyze the relationships between the sequence 
$
\{\delta^{(k)}_n\}
$
given in \eqref{deltas_aes} and the eigenvalues and eigenpolynomials of $L$.
In order to do that, we will prove Theorems \ref{lema1}-\ref{teorema_coeficientes} and, as an example, we will apply these results to the Hermite polynomials. Section \ref{seccion3} is devoted to the study of the linear transformations \eqref{999} of the sequence of eigenpolynomials. In this section, several auxiliary results and Theorem \ref{teorema1} are proved. Section \ref{seccion3} is applied to the Hermite polynomials in Section \ref{seccion4}, giving in this section some more lemmas and theorems related with condition \eqref{la_del_teorema1}. Finally, in Section \ref{seccion5} we summarize the main conclusions of this work.

\section{Eigenvalues and eigenpolynomials of $L$ }\label{seccion2}
We start this section proving Theorems \ref{lema1}-\ref{teorema_coeficientes}.

\noindent
{\bf Proof of Theorem \ref{lema1}:}

Firstly, we prove the equivalence between \eqref{11} and \eqref{7}. We underline that \eqref{7} is a very important relationship between the sequence 
$
\{\delta_{n}^{(k)}\}
$
and the sequences of eigenvalues and eigenpolynomials satisfying \eqref{11}. We point out that \eqref{7} can be written as
\begin{equation}
\sum_{k=0}^N\left[\sum_{i=k}^{m+k} {m+k \choose i} i!a_{i,i-k}\right]b_{n,m+k}=\lambda_nb_{n,m},\quad n\geq 0\,,\quad m=0,1,\ldots, n\,,
\label{3}
\end{equation}
which in the case $N=2$ is a straightfull interpretation of (11) in \cite{Bochner}. We follow here the lines of this relevant paper of Bochner \cite{Bochner}. For each $n\in \mathbb{N}$, \eqref{111} leads to
\begin{eqnarray*}
\partial_x^{i}P_n(x)& = & \sum_{s=1}^{n-i+1}(i+s-1)(i+s-2)\cdots sb_{n,i+s-1}x^{s-1}\\
& = & \sum_{s=1}^{n-i+1}{i+s-1 \choose i}i!b_{n,i+s-1}x^{s-1},\quad i=1,\ldots, N,
\end{eqnarray*}
where we recall that $b_{n,j}=0$ if $j>n$ and we understand that the sum is equal to 0 if $i>n$. Then, \eqref{11} is equivalent to
\begin{equation}
\sum_{i=1}^N\left(\sum_{k=0}^i a_{i,k}x^k\right) \left(\sum_{s=1}^{n-i+1}{i+s-1 \choose i}i!b_{n,i+s-1}x^{s-1}\right)=\lambda_n\sum_{r=0}^n b_{n,r}x^r.
\label{22}
\end{equation}
Comparing the coefficients of $x^r, r=0,\ldots,$ in both sides of \eqref{22} we arrive to \eqref{3}, as we wanted to prove. 

In the second place, we prove \eqref{otro11}. With this purpose, we consider \eqref{7}, which is, for each fixed $n$,
\begin{equation}
\label{matriz}
(\delta^{(0)}_{m}-\lambda_n)b_{n,m}+\delta^{(1)}_{m+1}b_{n,m+1}+\cdots+ \delta^{(N)}_{m+N}b_{n,m+N}=0\,,\quad m=0,1,\ldots , n.
\end{equation}
The relation \eqref{matriz} can be interpreted as an upper triangular linear system, with matrix of coefficients $M_{n+1}$, whose unknowns are $ b_{n,0},\cdots, b_{n,n-1},b_{n,n}$. In matrix notation this is \eqref{otro11}. Hence the coefficients of $b_n$ define an eigenvector of $M_{n+1}$ corresponding to the eigenvalue $\lambda_n$.

The identity \eqref{****} together with the fact that 
$
\det(M_n-\lambda_n I_n)\neq 0\,$ prove that
the coefficients $ b_{n,0},\cdots, b_{n,n-1}$ are uniquely determined. 
\hfill $\square$

Next, we prove Theorem \ref{teorema_coeficientes}, which shows how the sequence 
$
\{\delta^{(k)}_n\}
$
provides the explicit expression of the coefficients of eigenpolynomials.

\noindent
{\bf Proof of Theorem \ref{teorema_coeficientes}:}

Taking $i=n-m$ in \eqref{matriz} we see 
$$
\delta^{(0)}_{m}b_{m+i,m}+\delta^{(1)}_{m+1}b_{m+i,m+1}+\cdots+ \delta^{(N)}_{m+N}b_{m+i,m+N}=\lambda_{m+i}b_{m+i,m},\quad i=0,1,\ldots,
$$
where, as usual, 
$
b_{m+i,m+j}=0
$
when $j>i$. Then, for $m$ fixed and $i=1,2,\ldots$ we have 
\begin{equation}
\left.
\begin{array}{cccccccccccc}
(\lambda_{m+1}-\lambda_m)b_{m+1,m}& &&&&&& = & \delta_{m+1}^{(1)}\\
(\lambda_{m+2}-\lambda_m)b_{m+2,m}& - &  \delta_{m+1}^{(1)}b_{m+2,m+1}&&&&& = & \delta_{m+2}^{(2)}\\
\vdots & & \vdots &&&&&& \vdots\\
(\lambda_{m+N}-\lambda_m)b_{m+N,m}& - &  \delta_{m+1}^{(1)}b_{m+N,m+1}&-&\cdots & -& \delta_{m+N-1}^{(N-1)}b_{m+N,m+N-1}&= & \delta_{m+N}^{(N)}\\
\vdots & & \vdots &&&&\vdots && \vdots\\
(\lambda_{s}-\lambda_m)b_{s,m}& - &  \delta_{m+1}^{(1)}b_{s,m+1}&-&\cdots & -& \delta_{m+N-1}^{(N-1)}b_{s,m+N-1}&= & \delta_{m+N}^{(N)}b_{s,m+N}\\
\vdots & & \vdots &&&&\vdots && \vdots\\
\end{array}
\right\}
\label{nuevo2}
\end{equation}
Hence, the first relation in \eqref{nuevo2} leads to
$$
b_{m+1,m}=\frac{ \delta_{m+1}^{(1)}}{\lambda_{m+1}-\lambda_m},
$$
which is \eqref{coeficientes} when $n=m+1$ and $i=m$.\\
From the second and consecutive relations of \eqref{nuevo2} we are going to obtain \eqref{coeficientes} for $n=m+k,\,i=m,$ and $k=2,3,\ldots$ In fact, assume that
$b_{m+k,m}$ satisfies \eqref{coeficientes} for $k=2,3,\ldots, q$. Then, the $(q+1)$-th relation of \eqref{nuevo2} implies
\begin{equation}
\label{nuevo3}
(\lambda_{m+q+1}-\lambda_m)b_{m+q+1,m}=\delta_{m+1}^{(1)}b_{m+q+1,m+1}+\delta_{m+2}^{(2)}b_{m+q+1,m+2}+\cdots+\delta_{m+N}^{(N)}b_{m+q+1,m+N}.
\end{equation}
 
On the other hand, since $m$ can be substituted in \eqref{nuevo2} by any $m+p,\,p\in \mathbb N$, we know that
$$
b_{m+q+1,m+j}=\mathlarger\sum_{E_{m+q+1}^{(m+j)}}\left(\prod_{s=1}^k\frac{\delta^{(i_s)}_{m+j+i_1+\cdots +i_s}}{\lambda_{m+q+1}-\lambda_{m+j+i_1+\cdots +i_{s-1}}}\right),\quad j=1,\ldots , N.
$$
From this and \eqref{nuevo3},
\begin{equation}
b_{m+q+1,m}=\mathlarger\sum_{j=1}^{q+1}\frac{\delta^{(j)}_{m+j}}{\lambda_{m+q+1}-\lambda_{m}}\mathlarger\sum_{E_{m+q+1}^{(m+j)}}\left(\prod_{s=1}^k\frac{\delta^{(i_s)}_{m+j+i_1+\cdots +i_s}}{\lambda_{m+q+1}-\lambda_{m+j+i_1+\cdots +i_{s-1}}}\right),\quad j=1,\ldots , N.
\label{nuevo5}
\end{equation}
Furthermore it is obvious that, if  
$
\sum_{s=1}^ki_s=q+1,
$
then 
$
\sum_{i_s\neq j}i_s=q-j+1
$
for each 
$
j\in \{i_1,\ldots, i_k\}.
$
Thus, it is easy to see
$$
E_{m+q+1}^{(m)}=\bigcup_{j=1}^{q+1} E_{m+q+1}^{(m+j)}.
$$
Consequently, from \eqref{nuevo5} we arrive to
\begin{eqnarray*}
b_{m+q+1,m} &= &\mathlarger\sum_{j=1}^{q+1}\mathlarger\sum_{E_{m+q+1}^{(m+j)}}\frac{\delta^{(j)}_{m+j}}{\lambda_{m+q+1}-\lambda_{m}}\left(\prod_{s=1}^k\frac{\delta^{(i_s)}_{m+j+i_1+\cdots +i_s}}{\lambda_{m+q+1}-\lambda_{m+j+i_1+\cdots +i_{s-1}}}\right)\\
& = & \mathlarger\sum_{E_{m+q+1}^{(m)}}
\left(\prod_{s=1}^k\frac{\delta^{(i_s)}_{m+i_1+\cdots +i_s}}{\lambda_{m+q+1}-\lambda_{m+i_1+\cdots +i_{s-1}}}\right),
\end{eqnarray*}
which is \eqref{coeficientes} for $n=m+q+1$ and $i=m$. \hfill $\square$

\begin{rem}
\label{remark 1}
Because
$
\delta_n^{(k)},\,k=0,1,\ldots ,n,
$
are polynomials in $n$, from \eqref{coeficientes} we recover the well-known fact that the coefficients 
$
b_{n,j}
$
are rational functions of the variable $n$.
\end{rem}

As an example, we consider the case of Hermite polynomials, which we denote as
$
\{H_n\}.
$ 
The three-term recurrence relation satisfied by these polynomials is 
$$
	\begin{cases}
		H_n=xH_{n-1}-\frac{1}{2}(n-1)H_{n-2},\ n\in\mathbb{N}\\
		H_0=1,\quad H_{-1}=0,
	\end{cases}	
$$
and the banded matrix associated to this recurrence relation is the tridiagonal matrix
\begin{equation}
\label{matrizHermite}
J=
\left(
\begin{array}{lllccccccc}
0&1&0&\cdots  \\ 
 \frac{1}{2}& 0&1&\ddots \\ 
0&1&0&1&\ddots&\\ 
\vdots  &0&\frac{3}{2}&  \ddots&  \ddots&\\
 &&&\ddots &\ddots& 
 \end{array}
\right).
\end{equation}
It is well-known \cite{Bochner} that the Hermite polynomials are the classical eigenpolynomials of  the differential operator 
$$
L\equiv a_1(x)\partial_x +a_2(x)\partial_x^2
$$
where
$a_1(x)=-2x$ and $a_2(x)=1$. With some straight computations,  we obtain
$$
\delta_r^{(k)}=\left\{
\begin{array}{cccc}
-2r & \text{if} &k=0\\
0 & \text{if} &k=1\\
r(r-1) & \text{if} &k=2\\
0 & \text{if} &k\geq 3
\end{array}
\right. ,\qquad r=k,k+1,\ldots 
$$
The eigenvalues are 
$
\delta_n^{(0)}=\lambda_n=-2n,\,n=1,2,\ldots,
$
and Theorem \ref{teorema_coeficientes} is applicable because the restriction \eqref{autovectores} is fulfilled.

Due to $\delta_r^{(1)}=0$ and $E_n^{(n-1)}=\{1\}$, in  \eqref{coeficientes} we see $b_{n,n-1}=0 $. Also, when $n-i$ is odd, there exists an index
$i_s$
odd in each 
$
(i_1,\ldots ,i_k)\in E_n^{(i)}
$
and $b_{n,i}=0 $. Regarding the even coefficients 
$b_{m+2s,m}, m,s,\in \mathbb N$, since 
$E_{m+2s}^{(m)}$
can be substituted by 
$\{\stackrel{(s)}{(2,\ldots,2)}\},$
we have

$$
b_{m+2s,m}=\prod_{j=1}^s\frac{\delta_{m+2j}^{(2)}}{\lambda_{m+2s}-\lambda_{m+2j-2}}=\frac{(-1)^s}{2^{2s}}\frac{(m+2s)!}{m!s!}
$$
or, what is the same,

\begin{equation}
b_{n,n-2s}=\frac{(-1)^s}{2^{2s}}\frac{n!}{(n-2s)!s!}.
\label{31tilde}
\end{equation}
Hence, we arrive to the known expression for the monic Hermite polynomials

$$
H_n(x)=x^n+\sum_{s=1}^{[\frac{n}{2}]}b_{n,n-2s}x^{n-2s}=\frac{n!}{2^n}\sum_{s=0}^{[\frac{n}{2}]}\frac{(-1)^s}{(n-2s)!s!}(2x)^{n-2s}.
$$

\section{Linear transformations on eigenpolynomials}\label{seccion3}

Given a banded matrix $J$ as in \eqref{otro13}, let $C\in\mathbb{C}$ be such that the determinants of truncation of the infinite matrix $CI-J$ formed by its first $n$ rows and columns satisfy
${\rm det} (CI_n-B_n)\neq 0\ \text{for each}\ n\in\mathbb{N}.$ 
In these conditions, it is well known \cite{MaribelPaco} the existence of a lower triangular matrix $L$ and an upper triangular matrix $U$ such that 
$$
J-CI=UL.
$$
In the case that concerns us, $L$ is a $(p+1)$-banded matrix whose entries in the diagonal we assume equal to 1.
We know that there exist $p$ bidiagonal matrices $L^{(i)},\,i=1,2,\ldots, p,$ such that $L=L^{(1)}L^{(2)}\cdots L^{(p)}$ (see \cite{Dani}), where
$$U=\left(\begin{array}{ccccc}
		\gamma_1 & 1 & & &\\
		& \gamma_{p+1} & 1 & &\\
		& & \gamma_{2p+1} & 1 &\\
		& & & \ddots & \ddots
	\end{array}\right),\, 
	L^{(i)}=\left(\begin{array}{ccccc}
		1 & & & &\\
		\gamma_{i+1} & 1 & & &\\
		& \gamma_{p+i+1} & 1 & &\\
		& & \gamma_{2p+i+1} & 1 &\\
		& & & \ddots & \ddots
	\end{array}\right),\, i=1,\ldots, p.$$
The product 
$
UL^{(1)}L^{(2)}\cdots L^{(p)}
$  
is called {\em bidiagonal Geronimus factorization} of $J$ and was introduced in \cite{Geronimus}. These factorizations together with the ones introduced in \cite{Dani} constitute the so-called {\em bidiagonal Darboux factorization}. We call {\em Geronimus transformation} of $J$ to each $(p+2)$-banded matrix
\begin{equation}
J^{(s)}:=CI+L^{(p-s+1)}L^{(p-s+2)}\cdots L^{(p)}UL^{(1)}\cdots L^{(p-s)},\quad s=1,2,\ldots, p.
\label{3636}
\end{equation}
For each of these matrices
$
J^{(s)}
$ 
there exists an associated sequence 
$\{P^{(s)}_n\}$
of polynomials verifying a $(p+2)$-term recurrence relation. These polynomials are called also {\em Geronimus transformed} of
$\{P_n\}$ (see \cite{Geronimus}).

In the following we analyze the relation between the Geronimus transformation and the linear transformation \eqref{999}, where we assume
$
\gamma_n\neq 0
$
for each
$n\in \mathbb{N}.$
We write \eqref{999} as
$
v^{(1)}(x)=Tv(x),
$
where 
$
v^{(1)}(x)=\left(P^{(1)}_0(x),P_1^{(1)}(x),\ldots \right)^T,
$
$
v(x)=\left(P_0(x),P_1(x),\ldots \right)^T
$
and
\begin{equation}
\label{defT}
T=
\left(
\begin{array}{rcccccccccc}
	1\\
	\gamma_1&1&  \\ 
	0 & \gamma_2 & \ddots \\
	&\ddots&\ddots& \\
\end{array}
\right).	
\end{equation}
Then 
$T^{-1}$
is an infinite lower triangular matrix and
$
v(x)=T^{-1}v^{(1)}(x).
$
We underline in this point the formal sense of 
$T^{-1}$.
This matrix does not necessarily represent an operator, we just understand 
$T^{-1}$
as a table of values satisfying that the formal product of $T$ times $T^{-1}$ is 
$TT^{-1}=T^{-1}T=I$. 
Since $T$ is a bidiagonal matrix, each entry of 
$TT^{-1}$
and 
$T^{-1}T$
is obtained from a sum with a finite number of terms. 

Using \eqref{otro13} and \eqref{999},
$$
xv^{(1)}(x)=TJv(x)=TJT^{-1}v^{(1)}(x)
$$
or, equivalently,
\begin{equation}
\left( 
TJT^{-1}-xI
\right)v^{(1)}(x)=0.
\label{otro36}
\end{equation}
In general, it is possible that $TJT^{-1}$ is not a banded Hessenberg matrix and, in that case,
$
\{P_n^{(1)}\}
$
does not satisfy a $(p+2)$-term recurrence relation for any $p\in \mathbb{N}$. However, if
\begin{equation}
J=CI+UL^{(1)}\cdots L^{(p)}
\label{otro36bis}
\end{equation}
is a bidiagonal Darboux factorization of $J$ then
$$
TJT^{-1}=CI+TUL^{(1)}\cdots L^{(p)}T^{-1}.
$$
Since 
$
TUL^{(1)}\cdots L^{(p-1)}
$
is a $(p+2)$-banded Hessenberg matrix, we have that the sequence of monic polynomials
$
\{P_n^{(1)}\}
$
satisfy a $(p+2)$-term recurrence relation if and only if 
$
L^{(p)}T^{-1}
$
is a diagonal matrix, that is, 
$
T=L^{(p)}.
$
Under these conditions,
$
TJT^{-1}=CI+TUL^{(1)}\cdots L^{(p-1)}
$
coincides with \eqref{3636} when $s=1$, that is, 
$TJT^{-1}=J^{(1)}$.
In other words, 
$
\{P_n^{(1)}\}
$
would be the sequence of polynomials corresponding to the first Geronimus transformation of
$
\{P_n\}.
$
This is the case of the Hermite polynomials, as we will show in Theorem \ref{teorema2_2} of Section \ref{seccion4}.

Given the sequence
$
\{P_n
\}
$ 
of eigenpolynomials of the operator in \eqref{1}, we are interested in studying whether the new polynomials 
$
\{P_n^{(1)}
\}
$
defined in \eqref{999} are eigenfunctions for some finite order differential operator like \eqref{1}. We will prove in this section Theorem \ref{teorema1}, for which we need before some auxiliary results. Firstly, the following lemma gives a relation between the eigenvalues 
$\{\lambda_n\}$
of 
$L$.

\begin{lem}
\label{lema2bis}
For 
$i,\,j \in \mathbb{N}$ and $i<j$ we have
\begin{equation}
\label{autos}
\lambda_j-\lambda_i=\sum_{s=1}^j\left[
{i \choose s-1}+{i+1 \choose s-1}+\cdots {j-1 \choose s-1}
\right]s! a_{s,s},
\end{equation}
where we understand 
$
{m\choose s-1}=0
$
for 
$m<s-1$.
\end{lem} 

\begin{proof}
Assume 
$i\in \mathbb{N}$
and
$j=i+1$. Then, 
using \eqref{2} and
$
{i \choose s}+ {i\choose s-1}={i+1\choose s}
$,
we have
\begin{eqnarray}
\label{este}
\lambda_{i+1}-\lambda_i&=&\sum_{s=1}^{i+1}
{i+1 \choose s}s! a_{s,s}-\sum_{s=1}^{i}{i \choose s}s! a_{s,s}\nonumber\\
&=&(i+1)!a_{i+1,i+1}+\sum_{s=1}^{i}\left[{i+1 \choose s}- {i\choose s}\right]s! a_{s,s}=\sum_{s=1}^{i+1}
{i \choose s-1}s! a_{s,s}
\end{eqnarray}

Take now $i,j\in \mathbb{N}$ and $j>i+1$. If we write 
$
\lambda_j-\lambda_i=(\lambda_j-\lambda_{j-1})+\cdots +(\lambda_{i+1}-\lambda_i)
$, we can use \eqref{este} in each difference to reach  \eqref{autos}. 
\end{proof}	

Also in the study of the transformed polynomials 
$
\{P_n^{(1)}\},
$ the sequence 
$
\{\delta_{n}^{(k)}\}
$
defined in \eqref{deltas_aes} plays an important role. Firstly, we show that it is possible to express such sequence explicitly in terms of the eigenvalues and the eigenpolynomials of $L$.

\begin{lem}
\label{lema3}
For $n\in \mathbb{N}$ and $k=1,\ldots,n,$ we have
$
(-1)^k\delta_n^{(k)}=$
\begin{equation}
\label{1*}
\left|
\begin{array}{cccccc}
 (\lambda_{n-k}-\lambda_{n-k+1}) b_{n-k+1,n-k}& (\lambda_{n-k}-\lambda_{n-k+2}) b_{n-k+2,n-k}&\cdots&\cdots  &(\lambda_{n-k}-\lambda_{n}) b_{n,n-k}\\ 
1& b_{n-k+2,n-k+1}&\cdots &\cdots& b_{n,n-k+1}\\ 
0 & 1&   \cdots& \cdots &b_{n,n-k+2} &\\
\vdots  & 0  & \ddots &&\vdots \\
\vdots  &\vdots & \ddots &\ddots &\vdots \\
0 & 0 &\cdots & 1& b_{n,n-1}
 \end{array}
\right|.
\end{equation}

\end{lem}
\begin{proof}
	The determinant on the right side of \eqref{1*} is of order $k$. We proceed by induction on $n$.
	If $k=1$, applying \eqref{matriz} (with $m=n-1$), we obtain
	$$
	 (\delta_{n-1}^{(0)}-\lambda_{n}) b_{n,n-1}+\delta_{n}^{(1)}=0,
	$$
	which is equivalent to
	$$
	(\lambda_{n-1}-\lambda_{n})b_{n,n-1}=-\delta_{n}^{(1)},
	$$
	that is exactly \eqref{1*}. We have just proved that \eqref{1*} holds for all $n\in \mathbb{N}$ and $k=1$. In particular, \eqref{1*} is true for $n=1$.
	
For any $n\in\mathbb{N}$ and $k=1,2,\ldots, n$, let 
	$G_n^{(k)}$ be the determinant on the right hand side of \eqref{1*}. Take a fixed $n\in\mathbb{N}$ and assume
	\begin{equation}
		\delta_m^{(k)}=(-1)^kG_m^{(k)}, \quad k=1,2,\ldots, m,
		\label{4}
	\end{equation}
	for each 
	$m=1,2,\ldots, n-1$. 
	We want to prove that \eqref{4} holds also for $m=n$. In fact, expanding 
	$G_n^{(k)}$
	when $k\leq n$ 
	along its last column and taking into account \eqref{4},
	$$
	G_n^{(k)}=    G_{n-1}^{(k-1)}b_{n,n-1}-G_{n-2}^{(k-2)}b_{n,n-2}+\cdots+(-1)^kG_{n-k+1}^{(1)}b_{n,n-k+1}+(-1)^{k+1}(\lambda_{n-k}-\lambda_n)b_{n,n-k}.$$
	Therefore
	$$
	(-1)^kG_n^{(k)}=-\left[\delta_{n-1}^{(k-1)}b_{n,n-1}+\delta_{n-2}^{(k-2)}b_{n,n-2}+\cdots+\delta_{n-k+1}^{(1)}b_{n,n-k+1}+(\lambda_{n-k}-\lambda_n)b_{n,n-k}\right]
	$$
	because 
	$k-i\leq n-i$ for $i=1,2,\ldots ,k-1$. From this and \eqref{matriz} we arrive to \eqref{4} for $m=n$, as we wanted to prove.
\end{proof}

Given any monic polynomial
$
Q_n(x)=x^n+q_{n,n-1}x^{n-1}+\cdots +q_{n,1}x+q_{n,0}
$
we define
\begin{equation}
\label{nuevo}
\begin{array}{lll}
\Delta_{k,n,s}(Q_n)=\qquad&&\\
\end{array}
\end{equation}
$$
\left|
\begin{array}{cccccc}
{n-k \choose s-1}q_{n-k+1,n-k}& \left[
{n-k \choose s-1}+{n-k+1 \choose s-1}
\right]q_{n-k+2,n-k}&\cdots&\cdots  &\left[
{n-k\choose s-1}+\cdots +{n-1 \choose s-1}
\right]q_{n,n-k}\\ 
1& q_{n-k+2,n-k+1}&\cdots &\cdots& q_{n,n-k+1}\\ 
0 & 1&   \cdots& \cdots &q_{n,n-k+2} &\\
\vdots  & 0  & \ddots &&\vdots \\
\vdots  &\vdots & \ddots &\ddots &\vdots \\
0 & 0 &\cdots & 1& q_{n,n-1}
 \end{array}
\right|.
$$
Furthermore, we denote
\begin{equation}
\label{otro_nuevo}
\Delta_{k,n,s}:=\Delta_{k,n,s}(P_n),\quad \Delta^{(1)}_{k,n,s}:=\Delta_{k,n,s}(P_n^{(1)})
\end{equation}

As an immediate consequence of Lemmas \ref{lema2bis} and \ref{lema3}, we obtain the following.

\begin{lem}
\label{lema4}
For $n\in \mathbb{N}$ and $k=1,2,\ldots , n$ we have
$$
(-1)^{k+1}\delta_n^{(k)}=\sum_{s=1}^n \Delta_{k,n,s}s!a_{s,s}.
$$
\end{lem}
Moreover, we prove in the following lemma that
$
\Delta^{(1)}_{k,n,s}$
and 
$
\Delta_{k,n,s}$
in \eqref{otro_nuevo} are related.

\begin{lem}
\label{lema5}
With the above notation,
$$
\Delta^{(1)}_{k,n,s}=\Delta_{k,n,s}+\sum_{j=0}^{k-1}(-1)^j {{n-k+j} \choose {s-1}}b_{n-k+j,n-k}\left[
\sum_{r=1}^{k-j}\gamma_{n-k+j+1}\cdots \gamma_{n-k+j+r} E_{k,n,j+r-1}
\right],
$$
where
$E_{k,n,j}$ 
is defined in \eqref{otro17}.
\end{lem}

\begin{proof}
Let us start denoting
$D_{k,n,k-1}^{(1)}=b_{n,n-k}^{(1)}$
and

	$$D_{k,n,j}^{(1)}=\left|\begin{array}{ccccc}
		b_{n-k+j+1,n-k}^{(1)} & b_{n-k+j+2,n-k}^{(1)} & \cdots & \cdots & b_{n,n-k}^{(1)}\\
		1& b_{n-k+j+2,n-k+j+1}^{(1)}&\cdots &\cdots& b_{n,n-k+j+1}^{(1)}\\ 
		0 & 1&   \cdots& \cdots &b_{n,n-k+j+2}^{(1)} \\
		\vdots  & 0  & \ddots &&\vdots \\
		\vdots  &\vdots & \ddots &\ddots &\vdots \\
		0 & 0 &\cdots & 1& b_{n,n-1}^{(1)}	
	\end{array}\right|, \quad j=0,1,\ldots ,k-2.
$$
If we expand the determinant that define $\Delta_{k,n,s}^{(1)}$ (see \eqref{nuevo}-\eqref{otro_nuevo}) along its first row, it is easy to see that 
	\begin{equation}
		\label{eq-aux-3}
		\Delta_{k,n,s}^{(1)}=\sum_{j=0}^{k-1}(-1)^j\binom{n-k+j}{s-1}D_{k,n,j}^{(1)}.
	\end{equation}
	Notice that each determinant $D_{k,n,j}^{(1)}$ is of order $k-j$. Furthermore, \eqref{999} implies that
$$
b_{n,i}^{(1)}=b_{n,i}+\gamma_nb_{n-1,i},\quad n\in \mathbb{N},\quad i=0,1,\ldots, n
$$
(where
$
b_{n,n}^{(1)}=b_{n,n}=1,\,b_{n-1,n}=0
$). Using this equality in the first column of $D_{k,n,j}^{(1)}$ and expanding this determinant as the sum of two determinants, we obtain
	\begin{align}
		\label{eq-aux-1}	
		D_{k,n,j}^{(1)}&=
\left|\begin{array}{cccccc}
					b_{n-k+j+1,n-k}& b_{n-k+j+2,n-k}^{(1)} & \cdots & \cdots & b_{n,n-k}^{(1)}\\
					1& b_{n-k+j+2,n-k+j+1}^{(1)}&\cdots &\cdots& b_{n,n-k+j+1}^{(1)}\\ 
					0 & 1&   \cdots& \cdots &b_{n,n-k+j+2}^{(1)} \\
					\vdots  & 0  & \ddots &&\vdots \\
					\vdots  &\vdots & \ddots &\ddots &\vdots \\
					0 & 0 &\cdots & 1& b_{n,n-1}^{(1)}	
				\end{array}\right|\nonumber\\\nonumber\\
		&+\gamma_{n-k+j+1}b_{n-k+j,n-k}
\left|\begin{array}{ccccc}
					b_{n-k+j+2,n-k+j+1}^{(1)}&\cdots &\cdots& b_{n,n-k+j+1}^{(1)}\\ 
					1&   \cdots& \cdots &b_{n,n-k+j+2}^{(1)} \\
					0  & \ddots &&\vdots \\
					\vdots & \ddots &\ddots &\vdots \\
					0 &\cdots & 1& b_{n,n-1}^{(1)}	
				\end{array}\right|
	\end{align}
	Repeating this procedure in the second column of the first addend on the right hand side of \eqref{eq-aux-1},	
	\begin{align*}D_{k,n,j}^{(1)}&=
\left|\begin{array}{cccccc}
					b_{n-k+j+1,n-k}& b_{n-k+j+2,n-k} & \cdots & \cdots & b_{n,n-k}^{(1)}\\
					1& b_{n-k+j+2,n-k+j+1}&\cdots &\cdots& b_{n,n-k+j+1}^{(1)}\\ 
					0 & 1&   \cdots& \cdots &b_{n,n-k+j+2}^{(1)} \\
					\vdots  & 0  & \ddots &&\vdots \\
					\vdots  &\vdots & \ddots &\ddots &\vdots \\
					0 & 0 &\cdots & 1& b_{n,n-1}^{(1)}	
				\end{array}\right|\\\\
&+\gamma_{n-k+j+1}b_{n-k+j,n-k}
\left|\begin{array}{ccccc}
					b_{n-k+j+2,n-k+j+1}^{(1)}&\cdots &\cdots& b_{n,n-k+j+1}^{(1)}\\ 
					1&   \cdots& \cdots &b_{n,n-k+j+2}^{(1)} \\
					0  & \ddots &&\vdots \\
					\vdots & \ddots &\ddots &\vdots \\
					0 &\cdots & 1& b_{n,n-1}^{(1)}	
				\end{array}\right|
	\end{align*}
because such addend can be expressed as the addition of two determinants of which the second one is zero. Iterating the procedure we see that the first addend on the right hand side of \eqref{eq-aux-1} is
	$$D_{k,n,j}:=\left|\begin{array}{cccccc}
		b_{n-k+j+1,n-k} & b_{n-k+j+2,n-k}& \cdots & \cdots & b_{n,n-k}\\
		1& b_{n-k+j+2,n-k+j+1}&\cdots &\cdots& b_{n,n-k+j+1}\\ 
		0 & 1&   \cdots& \cdots &b_{n,n-k+j+2} \\
		\vdots  & 0  & \ddots &&\vdots \\
		\vdots  &\vdots & \ddots &\ddots &\vdots \\
		0 & 0 &\cdots & 1& b_{n,n-1}
	\end{array}\right|,$$
	(We denote this term by $D_{k,n,j}$ because it has the same structure than $D_{k,n,j}^{(1)}$ but switching $b_{ij}^{(1)}$ by $b_{ij}$.) Going back to \eqref{eq-aux-1}, we denote the determinant of the second addend on the right hand side as $E_{k,n,j}^{(1)}$, following the similarity with the determinant $E_{k,n,j}$ defined in \eqref{otro17}. Then, we have proved
	\begin{equation}
		\label{eq-I}
		D_{k,n,j}^{(1)}=D_{k,n,j}+\gamma_{n-k+j+1}b_{n-k+j,n-k}E_{k,n,j}^{(1)},
	\end{equation}
	where $E_{k,k,k-1}^{(1)}=1$. Now, we can repeat with the second addend in \eqref{eq-I} the same idea that we used with $D_{k,n,j}$. That is, at each step we select a column, express all the coefficients $b_{ij}^{(1)}$ from that column as $b_{ij}^{(1)}=b_{ij}+\gamma_ib_{i-1,j}$ and split the determinant in two following the addition given by the former expression. Applying this process to the first column and recalling that $b_{n-k+j+1,n-k+j+1}=1$, we can express $E_{k,n,j}^{(1)}$ as
	\begin{equation}
		\label{eq-aux-2}
		E_{k,n,j}^{(1)}=E_{k,n,j}+\gamma_{n-k+j+2}E_{k,n,j+1}^{(1)}, \quad j=0,1,\dots,k-1.
	\end{equation} 
	Applying \eqref{eq-aux-2} to $E_{k,n,j+1}^{(1)}$ we can rewrite \eqref{eq-aux-2} itself as
	\begin{align*}
		E_{k,n,j}^{(1)}&=E_{k,n,j}+\gamma_{n-k+j+2}E_{k,n,j+1}+\gamma_{n-k+j+2}\gamma_{n-k+j+3}E_{k,n,j+2}^{(1)}\\
		=\cdots&=\sum_{r=1}^{k-j}(\gamma_{n-k+j+2}\gamma_{n-k+j+3}\cdots\gamma_{n-k+j+r})E_{k,n,j+r-1},
	\end{align*}
	where $\gamma_{n-k+j+2}\gamma_{n-k+j+3}\cdots\gamma_{n-k+j+r}=1$ if $r=1$, and $E_{k,n,k-1}=1$.
	Using this expression in \eqref{eq-I} we reach
\begin{equation}\label{otro42}
	D_{k,n,j}^{(1)}=D_{k,n,j}+b_{n-k+j,n-k}\sum_{r=1}^{k-j}(\gamma_{n-k+j+1}\gamma_{n-k+j+2}\gamma_{n-k+j+3}\cdots\gamma_{n-k+j+r})E_{k,n,j+r-1}.
\end{equation}
	Finally, we use \eqref{otro42} to rewrite \eqref{eq-aux-3} as
	\begin{align*}
	\Delta_{k,n,s}^{(1)}&=\sum_{j=0}^{k-1}(-1)^j\binom{n-k+j}{s-1}D_{k,n,j}\\
	&+\sum_{j=0}^k(-1)^j\binom{n-k+j}{s-1}b_{n-k+j,n-k}\left[\sum_{r=1}^{k-j}(\gamma_{n-k+j+1}\cdots\gamma_{n-k+j+r})E_{k,n,j+r-1}\right].
	\end{align*}
Since the first addend is exactly $\Delta_{k,n,s}$, the proof concludes here.
	\end{proof}

\noindent
{\bf Proof of Theorem \ref{teorema1}:}

Using \eqref{1} and \eqref{operador_nuevo}, we recall (see \eqref{2}) 
$$
\lambda_n=\sum_{i=1}^{n} 
{n \choose i} i!a_{i,i}=\sum_{i=1}^{n} 
{n \choose i} i!a^{(1)}_{i,i},\quad n\in \mathbb{N}.
$$
Then, taking 
$n=1,2,\ldots,$
we check
$
a_{i,i}=a_{i,i}^{(1)}
$
for each 
$i\in \mathbb{N}$. 

On the other hand, defining 
$$
\widetilde\delta^{(k)}_{n}=\sum_{i=k}^n {n \choose i}i!a^{(1)}_{i,i-k},\quad k=0,1,\ldots,n,
$$
(see \eqref{deltas_aes}) and applying Lemma \ref{lema4} to 
$L^{(1)}$, we obtain
$$
\widetilde\delta^{(k)}_{n} = (-1)^{k+1}\sum_{s=1}^n \Delta^{(1)}_{k,n,s}s!a_{s,s}.
$$
Then, using Lemma \ref{lema5} and again Lemma \ref{lema4}, 
\begin{eqnarray*}
\widetilde\delta^{(k)}_{n} &=&\\
=\delta^{(k)}_{n}&+&(-1)^{k+1}\sum_{s=1}^n \sum_{j=0}^{k-1}(-1)^j {{n-k+j} \choose {s-1}}b_{n-k+j,n-k}
\left[\sum_{r=1}^{k-j}\gamma_{n-k+j+1}\cdots \gamma_{n-k+j+r} E_{k,n,j+r-1}
\right]s!a_{s,s}\\
=\delta^{(k)}_{n}&+&(-1)^{k+1}\sum_{j=0}^{k-1}(-1)^j\left[\sum_{s=1}^n  {{n-k+j} \choose {s-1}}s!a_{s,s}\right]b_{n-k+j,n-k}
\sum_{r=1}^{k-j}\gamma_{n-k+j+1}\cdots \gamma_{n-k+j+r} E_{k,n,j+r-1}.
\end{eqnarray*}
Since
$
\{P_n^{(1)}
\}
$
is the sequence of eigenpolynomials for 
$L^{(1)}$,
then
$
\widetilde\delta^{(k)}_{n} =0$
for 
$n\geq k>\widetilde N$. 
Also,
$
\delta^{(k)}_{n} =0$ for
$n\geq k > N$
(see \eqref{deltas_aes}). Therefore, we arrive to \eqref{la_del_teorema1}.
\hfill $\square$

\section{Darboux transformations on the sequence of Hermite polynomials}\label{seccion4}
Theorem \ref{teorema1} gives a necessary condition for the Darboux transformation to produce a new sequence of eigenpolynomials for some finite order differential operator with the same sequence of eigenpolynomials as $L$. In this section we show that \eqref{la_del_teorema1} is not verified in the case of Hermite polynomials. This proves that Darboux transformations, even in the case of classical orthogonal polynomials, may not lead to new families of eigenfunctions.

With the purpose to analyze Theorem \ref{teorema1} when 
$
\{P_n(x)\}=\{H_n(x)\}$ 
are the Hermite polynomials, we study the sequence 
$\{\gamma_i\}$
that defines
$
\{P^{(1)}_n(x)\}=\{H^{(1)}_n(x)\}$ 
in this case. 
\begin{teo}
\label{teorema2_2}
Let 
$\{H_n(x)\}$ and $\{H^{(1)}_n(x)\}$ 
be sequences that satisfy \eqref{999}, that is,
\begin{equation}
\label{999_H}
H_n^{(1)}(x)=H_n(x)+\gamma_nH_{n-1}(x).
\end{equation}
Then 
$\{H_n^{(1)}(x)
\}$
satisfies a $(p+2)$-term recurrence relation as \eqref{otro9} if and only if \eqref{999_H} is a Geronimus transformation of
$\{H_n(x)\}$. In that case, $p=1$ and we have 
\begin{equation}
\label{otro43}
\gamma_m=\gamma_2+\frac{1}{2\gamma_1}-\frac{m-1}{2\gamma_{m-1}},\quad m=2,3,\ldots 
\end{equation}
\end{teo}
\begin{proof}
Let $T$ and $J$ be the matrices defined in \eqref{defT} and \eqref{matrizHermite}, respectively. 
Then, $TJT^{-1}$
is a Hessenberg matrix, that is, 
$$
TJT^{-1}=
\left(
\begin{array}{rcccccccccc}
\delta_{0,0}&1&0&\cdots  \\ 
 \delta_{1,0}&\delta_{1,1}&1&0&\cdots \\ 
\vdots  & \vdots &\ddots&  \ddots&  \ddots&\\
\delta_{s,0}& \delta_{s,1}&& \delta_{s,s}&1&0&\cdots \\ 
\vdots &\vdots &&\vdots &\ddots &\ddots&  \ddots&\\
 \end{array}
\right)
$$
where
$
\delta_{r+s,r}=\gamma_{r+s}\alpha_{r+s-1,r}+\alpha_{r+s,r}-\gamma_{r+1}\delta_{r+s,r+1}
$
for
$r,\,s=0,1,\ldots,$
and we understand 
$
\delta_{r,r+1}=1.
$
Here $\alpha_{i,j}$ represent the entries
of
$J$,
which are given in \eqref{matrizHermite}. Thus, the main diagonal of $TJT^{-1}$ is given as
$
\delta_{r,r}=\gamma_r-\gamma_{r+1},\, r=0,1,\ldots
$
(assuming 
$\gamma_0=0$). The first subdiagonal is 
$$
\delta_{r+1,r}=\frac{r+1}{2}-\gamma_{r+1}\delta_{r+1,r+1}=\frac{r+1}{2}-\gamma_{r+1}\left(\gamma_{r+1}-\gamma_{r+2}\right),\quad r=0,1,\ldots
$$
The second subdiagonal is
\begin{eqnarray}
\label{otroIII}
\delta_{r+2,r}&=&\gamma_{r+2}\frac{r+1}{2}-\gamma_{r+1}\frac{r+2}{2}+\gamma_{r+1}\gamma_{r+2}\left(\gamma_{r+2}-\gamma_{r+3}\right),\quad r=0,1,\ldots,\nonumber
\end{eqnarray}
and the following subdiagonals are 
\begin{equation}
\label{otroIV}
\delta_{r+s,r}=-\gamma_{r+1}\delta_{r+s,r+1},\quad r=0,1,\ldots, \quad s\geq 3.
\end{equation}
As in \eqref{otro36}, we see that
$$
\left( 
TJT^{-1}-xI
\right)
\left(
\begin{array}{c}
H_0^{(1)}(x)\\
H_1^{(1)}(x)\\
\vdots
\end{array}
\right)=0.
$$
If 
$TJT^{-1}$ is the banded matrix corresponding to a $(p+2)$-term recurrence relation for some $p\in \mathbb{N}$ then 
\begin{equation}
\label{otroV}
\sum_{k=n-p}^{n-1}\delta_{n,k} H^{(1)}_k(x)+\left( \delta_{n,n}-x \right)H^{(1)}_n(x)+H^{(1)}_{n+1}(x)=0,\quad n=0,1,\ldots,
\end{equation}
with
$
\delta_{n,n-p}\neq 0$ for $n=p,p+1,\ldots 
$ That is,
$
\delta_{n,n-p-1}= 0$ for $n=p,p+1,\ldots ,
$
and also 
\begin{equation}
\label{contrario}
\delta_{n,n-p-s}=0,\ \text{for}\ n=p,p+1,\ldots,\, s\geq 1. 
\end{equation}
Under these conditions
$p\geq 1$, 
because 
$p=0$ 
implies
$
\delta_{n,n-1}=0,$
that is,
$
\frac{n}{2}=\gamma_n(\gamma_n-\gamma_{n+1}),\, n=1,2,\ldots,
$
and then we would have 
$
\delta_{n,n-2}\neq 0
$
in \eqref{otroIII}, which contradicts \eqref{contrario}.

On the other hand, \eqref{otroIV} indicates that it is not possible to have
$
\delta_{n,n-p-1}= 0,
$
if
$
\delta_{n,n-p}\neq 0
$
for $n=p+1,p+2,\ldots$, $p\geq 2$. 
Consequently, 
$p=1$ and \eqref{otroV} is a three-term recurrence relation. 

Let
$J=CI+UL$
be a Geronimus factorization of $J$ (see \eqref{otro36bis}). Since $TJT^{-1}$ and $TU$ are tridiagonal matrices, $LT^{-1}$ must be a diagonal matrix. In addition, the entries in the diagonals of 
$L$
and
$T^{-1}$
coincide, so we  have 
$L=T$. Therefore,
$
TJT^{-1}=CI+LU
$
is a Darboux (Geronimus) transformation of $J$.

Reciprocally, if \eqref{999_H} is a Geronimus transformation of 
$
\{H_n(x)\}
$
obviously 
$
TJT^{-1}=CI+TU
$
is a tridiagonal matrix and the polynomials
$
\{H^{(1)}_n(x)\}
$
satisfy a three-term recurrence relation. 

Finally, we assume that \eqref{999_H} holds. Then \eqref{otroIII} provides
$$
0=\gamma_{r+2}\frac{r+1}{2}-\gamma_{r+1}\frac{r+2}{2}+\gamma_{r+1}\gamma_{r+2}\left(\gamma_{r+2}-\gamma_{r+3}\right),\quad r=0,1,\ldots,
$$
that is,
\begin{equation}
\label{otroVI}
\gamma_{r+3}=\gamma_{r+2}+\frac{r+1}{2\gamma_{r+1}}-\frac{r+2}{2\gamma_{r+2}}\quad r=0,1,\ldots
\end{equation}
Note that \eqref{otroVI} is equivalent to 
$$
\gamma_{m}+\frac{m-1}{2\gamma_{m-1}}=\gamma_{2}+\frac{1}{2\gamma_{1}},\quad m=3,4,\ldots,
$$
as we wanted to prove. 
\end{proof}

We recall that in a Geronimus factorization we need to fix two parameters. In fact, firstly we choose 
$C\in \mathbb{C}$ such that there exists the 
$UL$
factorization of
$
J-CI.
$ 
Then, 
$UL$ 
depends on a second parameter. This implies that in \eqref{otroVI} the sequence 
$
\{\gamma_i\}
$
is determined in terms of
$
\gamma_1,\,\gamma_2.
$
In the sequel we assume 
$$
\gamma_{2}+\frac{1}{2\gamma_{1}}=0.
$$
Thus, since \eqref{otro43}, if \eqref{999_H} is a Geronimus transformation of 
$
\{H_n\}
$
we have 
\begin{equation}
\label{otro**}
\gamma_m\gamma_{m+1}=-\frac{m}{2},\quad m=1,2,\ldots,
\end{equation}
and it is easy to check
\begin{equation}
\label{otro*}
\left.
\begin{array}{lcc}
\gamma_{2m} & = & -\displaystyle\frac{(2m-1)(2m-3)\cdots 5\cdot 3}{(2m-2)(2m-4)\cdots 4\cdot 2}\frac{1}{2\gamma_1}\\\\
\gamma_{2m+1} & = & \displaystyle\frac{2m(2m-2)\cdots 4\cdot 2}{(2m-1)(2m-3)\cdots 5\cdot 3}\gamma_1
\end{array}
\right\},\quad m=1,2,\ldots
\end{equation}
(writing $(2m-2)(2m-4)\cdots 4\cdot 2=1,\, (2m-1)(2m-3)\cdots 5\cdot 3=1,$
when $m=1$).

In the case of Hermite polynomials, 
$a_{11}=-2, a_{ii}=0, \forall i\geq 2$. 
Moreover, the coefficients of the polynomials were given in \eqref{31tilde}. Then \eqref{la_del_teorema1} becomes
$
-2\Sigma_H(n,k),
$
where
\begin{equation}
	\label{eq-IIa}
	\Sigma_H(n,k):=\sum_{\substack{j=0\\ j\ \text{even}}}^{k-1}
	\frac{(-1)^{j/2}}{2^j}\frac{(n-k+j)!}{(n-k)!(j/2)!}\sum_{r=1}^{k-j}\gamma_{n-k+j+1}\ldots \gamma_{n-k+j+r} E_{k,n,j+r-1}.
\end{equation}
For the rest of this section, our goal will be to prove that \eqref{la_del_teorema1} is not satisfied. In order to do this, we will show that 
$
\Sigma_H(n,k)\neq 0 
$
for $n,\,k$
under the conditions of Theorem \ref{teorema1}. 

To compute the right hand side of \eqref{eq-IIa},  we will use the following result.
\begin{lem}
\label{lema8}
Let us define
$$
S_M(m)=\sum_{r=0}^M\left(-\frac{1}{4}\right)^r\frac{(2m+2r-1)!}{r!(m+r-1)!}{m-1+M\choose m-1+r},\quad m,\,M\in \mathbb{N}.
$$
Then 
\begin{eqnarray}
S_M(1) & = &\displaystyle\frac{(2\cdot 1-3)(2\cdot 2-3)\cdots (2M-3)}{M!2^M}\label{BBB}\\
S_M(m)& =&\displaystyle\frac{(2m-1)!}{(m-1)!}S_M(1) \label{AA} 
\end{eqnarray}
for $n,\,M\in \mathbb{N}$. In particular we have
$
S_M(m)\neq 0 
$
for all
$m,\,M\in \mathbb{N}.$
\end{lem}
\begin{proof}
Using
$$
{m-1+M\choose m-1+r}={m-2+M \choose m-1+r} +{m-2+M \choose m-2+r}
$$
it is easy to check
\begin{equation}
\label{otro1}
S_M(m)=2(2m-1)S_M(m-1),\quad m=1,2,\ldots, \,M=0,1,\ldots 
\end{equation}
Now, for each $m,\,M $ we write
\begin{equation}
\label{otro2}
S_M(m)=\frac{(2m-1)!}{(m-1)!}f(m,M)
\end{equation}
and since \eqref{otro1} we see 
$
f(m-1,M)=f(m,M).
$
This means that 
$
f(M):=f(m,M)
$
does not depend on $m$. Furthermore, from \eqref{otro2} it is
$S_M(1)=f(M)$
and \eqref{otro2} becomes \eqref{AA}.

On the other hand, writting
\begin{equation}
\label{222}
{M\choose r}={M-1 \choose r-1} +{M-1 \choose r},\quad r=0,1,\ldots,M
\end{equation}
(where we assume 
$
{M-1 \choose -1}={M-1 \choose M}=0
$
)
and substituting this expression in
$$
S_M(1)=\sum_{r=0}^M\left(-\frac{1}{4}\right)^r\frac{(2r+1)!}{(r!)^2}{M\choose r},
$$
we have
\begin{eqnarray}
\label{333}
S_M(1)&=&\sum_{r=1}^M\left(-\frac{1}{4}\right)^r\frac{(2r+1)!}{(r!)^2}{M-1\choose r-1}+\sum_{r=0}^{M-1}\left(-\frac{1}{4}\right)^r\frac{(2r+1)!}{(r!)^2}{M-1\choose r}\\
&=& \sum_{r=0}^{M-1}\left(-\frac{1}{4}\right)^r\frac{(2r+1)!}{(r!)^2}\left(-\frac{1}{2(r+1)} \right){M-1\choose r}=
 -\frac{1}{2}\sum_{r=0}^{M-1}\left(-\frac{1}{4}\right)^r\frac{(2r+1)!}{(r+1)!r!}{M-1\choose r}.\nonumber
\end{eqnarray}
We will show
\begin{equation}
\label{444}
S_M(1)=\frac{(2\cdot 1-3)(2\cdot 2-3)\cdots (2\cdot q-3)}{2^q}
\sum_{r=0}^{M-q}\left(-\frac{1}{4}\right)^r\frac{(2r+1)!}{(r+q)!r!}{M-q\choose r},\,q=1,\ldots ,M.
\end{equation}
In fact, \eqref{333} is \eqref{444} for $q=1$. Proceeding by induction, and considering 
$$
{M-q\choose r}={M-q-1 \choose r-1} +{M-q-1 \choose r}
$$
as in \eqref{222}, we obtain \eqref{444} for $q+1$ when we assume that \eqref{444} is satisfied for $q$. In particular, for $q=M$ in \eqref{444} we arrive to \eqref{BBB}.
\end{proof}

Also with the aim to study \eqref{eq-IIa}, in the following lemma we analyze the determinant $E_{k,n,j}$ defined in \eqref{otro17}.

\begin{lem}
	\label{lema6}
	If $E_{k,n,j+r-1}$ is considered over the Hermite polynomials, then
	$$E_{k,n,j+r-1} =\begin{cases}
		0, & \text{ if}\ k-j-r\ \text{is odd,}\\
		\displaystyle\frac{n!}{(n-k+j+r)!((k-j-r)/2)!2^{k-j-r}}, & \text{ if}\ k-j-r\ \text{is even.}
	\end{cases}$$
\end{lem}

\begin{proof}
The order of the determinant $E_{k,n,j+r-1}$ is $k-j-r$. In the first place, if $k-j-r$ is odd then $E_{k,n,j+r-1}$ has an odd number of columns. 
Taking into account \eqref{31tilde},
$E_{k,n,j+r-1}=$ 
$$=\left|\begin{array}{cccccc}
	0 & -\frac{1}{2^2}\frac{(n-k+j+r+2)!}{(n-k+j+r+1)!} & 0 & \frac{1}{2^4}\frac{(n-k+j+r+4)!}{(n-k+j+r)!2!} & \cdots & 0\\
	1 & 0 & -\frac{1}{2^2}\frac{(n-k+j+r+3)!}{(n-k+j+r+1)!} & 0 & \cdots & \frac{(-1)^{\frac{k-j-r-1}{2}}}{2^{k-j-r-1}}\frac{n!}{(n-k+j+r+1)!((k-j-r-1)/2)!}\\
	0 & \ddots & \ddots & \ddots &  & \vdots\\
	\vdots & \ddots & \ddots & \ddots & \ddots & \vdots\\
	\vdots & \ddots & \ddots & \ddots & \ddots & \vdots\\
	\vdots &  & 0 & 1 & 0 & -\frac{1}{2^2}\frac{(n-1)!}{(n-3)!}\\
	0 & \cdots & \cdots & 0 & 1 & 0
\end{array}\right|$$
Expanding this determinant each time along their first 
$(k-j-r-1)/2$
odd columns we obtain a determinant of order 
$(k-j-r+1)/2$
whose entries in the last column are zero. Then
$E_{k,n,j+r-1}=0$, which proves the case $k-j-r$ odd.

In the second place, if $k-j-r$ is even then
$E_{k,n,j+r-1}=$
$$=\left|\begin{array}{ccccccc}
	0 & \frac{-1}{2^2}\frac{(n-k+j+r+2)!}{(n-k+j+r)!} & 0 & \frac{1}{2^4}\frac{(n-k+j+r+4)!}{(n-k+j+r)!2!}& \cdots& \frac{(-1)^{\frac{k-j-r}{2}}}{2^{k-j-r}}\frac{n!}{(n-k+j+r)!((k-j-r)/2)!}\\
	1 & 0 & \frac{-1}{2^2}\frac{(n-k+j+r+3)!}{(n-k+j+r+1)!} & 0 & \cdots & 0\\
	0 & \ddots & \ddots & \ddots & & \vdots\\
	\vdots & \ddots & \ddots & \ddots &  \ddots & \vdots\\
	\vdots & & & 1 &  0 & -\frac{1}{2^2}\frac{n!}{(n-2)!}\\
	0 & \cdots & \cdots & \cdots & 1 & 0
\end{array}\right|$$
Expanding each time along the odd columns,
$E_{k,n,j+r-1}=$
$$=(-1)^{\frac{k-j-r}{2}}\left|\begin{array}{ccccc}
	\frac{-1}{2^2}\frac{(n-k+j+r+2)!}{(n-k+j+r)!} & \frac{1}{2^4}\frac{(n-k+j+r+4)!}{(n-k+j+r)!2!} & \cdots&\frac{(-1)^{\frac{k-j-r}{2}}}{2^{k-j-r}}\frac{n!}{(n-k+j+r)!((k-j-r)/2)!}\\
	1 & \frac{-1}{2^2}\frac{(n-k+j+r+4)!}{(n-k+j+r+2)!} & \cdots & \frac{(-1)^{\frac{k-j-r-2}{2}}}{2^{k-j-r-2}}\frac{n!}{(n-k+j+r+2)!((k-j-r-2)/2)!}\\
	 & 1 & \   & \vdots\\
	 &  & \ddots & \vdots\\
	 & &  & -\frac{1}{2^2}\frac{n!}{(n-2)!}	
\end{array}\right|.$$
Taking firstly a common denominator out of each row and then a common numerator out of each column, we have, after simplifying
$$E_{k,n,j+r-1}=\frac{(-1)^{\frac{k-j-r}{2}}n!}{2^{k-j-r}(n-k+j+r)!}
\left|
\begin{array}{cccccc}
	-1 &
	\frac{1}{2!} & -\frac{1}{3!}&\cdots& \cdots&
	\frac{(-1)^{\frac{k-j-r}{2}}}{((k-j-r)/2)!}\\[3mm]
	1 & -1 & \frac{1}{2!}&\cdots & \cdots& \frac{(-1)^{\frac{k-j-r-2}{2}}}{((k-j-r+2)/2)!}\\[3mm]
	0&1& -1& \ddots &   & \frac{(-1)^{\frac{k-j-r-4}{2}}}{((k-j-r+4)/2)!}\\[3mm]
	&0 &1& -1 & \ddots &\frac{(-1)^{\frac{k-j-r-6}{2}}}{((k-j-r+6)/2)!}\\
	& & \ddots& \ddots & \ddots & \vdots\\
	& && 0 & 1 & -1
\end{array}\right|.$$
Furthermore, it is easy to prove that the value of the determinant on the right hand side is
$
\frac{(-1)^{\frac{k-j-r}{2}}}{\left(\frac{k-j-r}{2}\right)!}
$
(expanding the determinant along its last column). With this, the case $k-j-r$ even is also proved.
\end{proof}

Using the above lemmas, the expression of $\Sigma_H(n,k)$ in \eqref{eq-IIa} is transformed into
\begin{equation}
		\label{eq-IIb}
\Sigma_H(n,k)=	\sum_{\substack{j=0\\ j\ \text{even}}}^{k-1}\frac{(-1)^{j/2}}{2^j}\frac{(n-k+j)!}{(n-k)!(j/2)!}\sum_{\substack{s=0\\ s\ \text{even}}}^{k-j-1}\gamma_{n-k+j+1}\cdots\gamma_{n-s}\frac{n!}{(n-s)!(s/2)!2^s}.
\end{equation}
If 
$
\{H^{(1)}_n\}
$
were a sequence of eigenpolynomials for some differential operator 
$L^{(1)}$ 
with the same sequence  
$\{-2n\}$
of eigenvalues corresponding to $\{H_n\}$ then Theorem \ref{teorema1} would imply that 
$
\Sigma_H(n,k)=0$
for some 
$\widetilde N\in \mathbb{N},\,\widetilde N\geq 2,$
and
$n\geq k>\widetilde N$ (see \eqref{operador_nuevo}). However, in the following theorem we prove that the former necessary condition does not hold.

\begin{teo}
Let
$n$ be even and $k$ be odd, $n\geq k\geq 3$. Then
$
\Sigma_H(n,k)\neq 0$.
\end{teo}

\begin{proof}
Let $n$ and $k$ be fixed numbers that satisfy the hypothesis. Then,
\eqref{eq-IIb} can be rewritten as
\begin{equation}
\Sigma_H(n,k)=\frac{n!}{(n-k)!}\sum_{\substack{j=0\\ j\ \text{even}}}^{k-1}\frac{(-1)^{j/2}}{2^j(j/2)!}\sum_{\substack{s=0\\ s\ \text{even}}}^{k-j-1}\left(\frac{\gamma_{n-k+j+1}}{n-k+j+1}\right)\left(\frac{\gamma_{n-k+j+2}}{n-k+j+2}\right)\cdots\left(\frac{\gamma_{n-s}}{n-s}\right)\frac{1}{2^s(s/2)!}
\label{sigma}
\end{equation}
where 
$
\left(\frac{\gamma_{n-k+j+1}}{n-k+j+1}\right)\left(\frac{\gamma_{n-k+j+2}}{n-k+j+2}\right)\cdots\left(\frac{\gamma_{n-s}}{n-s}\right)$
has an odd number $k-j-s$ of factors that can be grouped as
$$
\frac{\gamma_{n-k+j+1}}{n-k+j+1}\left(\frac{\gamma_{n-k+j+2}\gamma_{n-k+j+3}}{(n-k+j+2)(n-k+j+3)}
\right)\cdots \left(\frac{\gamma_{n-s-1}\gamma_{n-s}}{(n-s-1)(n-s)}
\right).$$
Due to \eqref{otro**}, the above product is equal to
$$
\frac{\gamma_{n-k+j+1}}{(n-k+j+1)F(j,s)}\left(-\frac{1}{2}\right)^{\frac{k-j-s-1}{2}},
$$
where
$
F(j,s):=(n-k+j+3)(n-k+j+5)\cdots(n-s).
$
Substituting in \eqref{sigma} and using this notation, 
\begin{eqnarray*}
\Sigma_H(n,k)&=&
\frac{n!}{(n-k)!}\left(-\frac{1}{2}\right)^{\frac{k-1}{2}}\sum_{\substack{j=0\\ j\ \text{even}}}^{k-1}\frac{1}{2^{j/2}(j/2)!}\left(\frac{\gamma_{n-k+j+1}}{n-k+j+1}\right)
\sum_{\substack{s=0\\ s\ \text{even}}}^{k-j-1}\frac{(-1)^{s/2}}{2^{s/2}(s/2)!F(j,s)}\\
& = & \frac{n!}{(n-k)!}\left(-\frac{1}{2}\right)^{\frac{k-1}{2}}\sum_{r=0}^{\frac{k-1}{2}}\frac{\gamma_{n-k+2r+1}}{2^r r!}
\sum_{q=0}^{\frac{k-1}{2}-r}\frac{(-1)^q}{2^q q!F(2r-2,2q)}.
\end{eqnarray*}
Since $n$ is even and $k$ is odd, necessarily $n\geq k+1$. Hence, applying \eqref{otro*},
$$
\gamma_{n-k+2r+1}=-\frac{(n-k+2r)(n-k+2r-2)\cdots 5\cdot 3}{(n-k+2r-1)(n-k+2r-3)\cdots 4\cdot 2}\left(\frac{1}{2\gamma_1}\right)=
-\frac{F(k-n,k-2r)}{2\gamma_1F(k-n-1,k-2r+1)}.
$$
Moreover,
$$
F(k-n-1,k-2r+1)F(2r-2,2q)=F(k-n-1,2q)=2^{\frac{n}{2}-q}\left(\frac{n}{2}-q\right)!
$$
and
$$
F(k-n,k-2r)=\frac{(n-k+2r)!}{2^{\frac{n-k-1}{2}+r}\left(\frac{n-k-1}{2}+r\right)!}.
$$
Therefore,
\begin{equation}
\label{ultima}
\Sigma_H(n,k)= \frac{n!}{\gamma_1(n-k)!2^n}
\sum_{r=0}^{\frac{k-1}{2}}\frac{(n-k+2r)!}{2^{2r}\left(\frac{n-k-1}{2}+r\right)! r!}
\sum_{q=0}^{\frac{k-1}{2}-r}\frac{(-1)^{q+1}}{q!\left(\frac{n}{2}-q\right)!}.
\end{equation}
On the other hand, it is easy to see, using induction on $s$, that
$$
\sum_{i=0}^{m-s}(-1)^i{m \choose i}=(-1)^{m-s}{m-1 \choose s-1},\quad 1\leq s\leq m.
$$
Considering this and 
$$
\sum_{q=0}^{\frac{k-1}{2}-r}\frac{(-1)^q}{q!\left(\frac{n}{2}-q\right)!}=\frac{1}{(n/2)!}\sum_{q=0}^{\frac{k-1}{2}-r}(-1)^q{n/2 \choose q},
$$
and using the notation of \eqref{BBB}-\eqref{AA}, in \eqref{ultima} we have
\begin{eqnarray*}
\Sigma_H(n,k)&=& \frac{n!(-1)^{\frac{k+1}{2}}}{\gamma_1(n/2)!(n-k)!2^{n}}
\sum_{r=0}^{\frac{k-1}{2}}\frac{(-1)^r(n-k+2r)!}{2^{2r}\left(\frac{n-k-1}{2}+r\right)! r!}
{\frac{n}{2}-1 \choose \frac{n-k-1}{2}+r }\\
&=&
\frac{n!(-1)^{\frac{k+1}{2}}}{\gamma_1(n/2)!(n-k)!2^{n}}S_{\frac{k-1}{2}}\left(\frac{n-k+1}{2}\right)\neq 0
\end{eqnarray*}
as we wanted to prove.
\end{proof}

\section{Conclusions }\label{seccion5}
In this paper, the sequences 
$
\{
\delta_n^{(k)}
\},\, n\in \mathbb{N},\, k=1,2,\ldots,n,
$
are introduced, which are an important connection between the coefficients of the polynomials 
$\{P_n\}$
 and those of 
$
\{a_n\}
$ defining a differential operator
$L$ 
as in \eqref{1}.
In the first place, given a differential operator $L$, this connection allows, under certain conditions, to guarantee the existence and uniqueness of their eigenvalues and eigenvectors, at the same time that it leads to the explicit expression of such eigenvalues and eigenfunctions.
In the second place, we have derived a necessary condition for a linear transformation to provide a new family of eigenpolynomials for some finite order operator. 
Finally, we have tested this necessary condition for the particular case of Geronimus transformations applied to Hermite polynomials, which led us to the conclusion that Geronimus transformed of Hermite eigenpolynomials are not eigenfunctions of any new differential operator.

\vspace{.5cm}
\noindent
{\bf Declaration of competing interest}

\vspace{.3cm}
There is no competing interest.

\end{document}